\documentclass[english,12pt]{article} 
\pdfoutput=1
\usepackage{authblk}
\usepackage[top=2cm,bottom=2cm,left=1cm,right=1cm]{geometry}
\usepackage{fontenc}
\usepackage[utf8]{inputenc}
\usepackage{afterpage}
\usepackage[dvipsnames]{xcolor}
\usepackage{newtxtext,commath}
\usepackage{amscd}
\usepackage{amsmath}
\usepackage{amsthm}
\usepackage{newtxmath}
\usepackage{pdfsync}
\usepackage{a4wide}
\usepackage{hyperref}
\usepackage{amsfonts}
\usepackage{color}
\usepackage{enumerate}
\usepackage{graphicx}
\usepackage{amsfonts}
\usepackage{ragged2e}
\usepackage{mathtools}
\usepackage{comment}

\newtheorem{theorem}{Theorem}[section]

\newtheorem{lemma}[theorem]{Lemma}

\newtheorem{definition}[theorem]{Definition}

\numberwithin{equation}{section}

\newcommand{\R}{\mathbb{R}}
\newcommand{\HH}{\mathcal{H}}
\newcommand{\LL}{\mathcal{L}}
\newcommand{\A}{\mathcal{A}}

\newcommand{\eps}{\varepsilon}
\newcommand{\Om}{\Omega}

\newcommand{\prt}{\partial}

\newcommand{\lmb}{\lambda}

\renewenvironment{thebibliography}[1]{
  \begin{oldthebibliography}{#1}
    \setlength{\itemsep}{0.2em}
    \setlength{\parskip}{0.2em}
}
{
  \end{oldthebibliography}
}

\title{$H^\infty$-control problem for a singular parabolic system with convection 
  }
\author[1,2]{Cristian Cazacu} 
\author[2]{Gabriela Marinoschi}
\author[1,2]{Teodor Rugină}
\affil[1]{Faculty of Mathematics and Computer Science, University of Bucharest, 14 Academiei Street, 
		010014 Bucharest, Romania.}
\affil[2]{Gheorghe Mihoc-Caius Iacob Institute of Mathematical
	Statistics and Applied Mathematics of the Romanian Academy\\
	050711 Bucharest, Romania \newline 
    Emails:  cristian.cazacu@fmi.unibuc.ro, gabimarinoschi@yahoo.com, teorugina@yahoo.com.}
    \date{}

\begin{document}
\maketitle

\begin{abstract}
    We study the $H^\infty-$control problem for an infinite dimensional parabolic system, with a convection term, perturbed by a singular inverse-square potential with control distributed in the interior of a domain, extending part of the results in \cite{marinoschi}.
\end{abstract}

\section{Introduction}

\quad The $H^\infty$-control technique is an important tool in control theory, offering powerful methods for designing robust feedback controllers that stabilize dynamical systems while ensuring desired performance levels in the presence of uncertainties and external disturbances. The idea of this method relies upon the concept of minimizing the amplification from exogenous inputs to the system outputs, as measured in the $L^2$-norm. Specifically, the goal is to determine a feedback controller such that the transfer function from the disturbance to the output has an $H^\infty$-norm less than a prescribed threshold. This formulation naturally leads to an optimization problem in the frequency domain.

The origins of $H^\infty$-control can be traced back to the seminal work of Zames \cite{zames}, who framed the problem as one of minimizing the induced norm of certain operators. Subsequent developments led to state-space formulations, notably in the contributions by Glover \cite{glover}, and the derivation of necessary and sufficient conditions for solvability via algebraic Riccati equations. These advancements were crucial for enabling practical computation and implementation of $H^\infty$ controllers in finite-dimensional settings. The extension of $H^\infty$-control method to infinite-dimensional systems, particularly those governed by partial differential equations, posed substantial mathematical challenges. These were addressed in the literature through the application of semigroup theory and operator-theoretic techniques (see, e.g., \cite{curtain, staffans, vanKeulen2, vanderSchaft}). A significant body of work has been developed to study boundary and distributed control problems for such systems, where we note particularly the hyperbolic case in \cite{BarbuHinf2, BarbuHinf3, BarbuHinf1}.

In the present work, we aim to complement the problem studied in  \cite{marinoschi} in the context of parabolic PDEs with inverse-square potentials of the form $\frac{\lambda} {|x|^2}$, where the author was concerned with the $H^\infty$-control problem for linear infinite-dimensional parabolic systems involving singular potentials of Hardy type. Here, our study extends part of the results in \cite{marinoschi} by adding  a convective term to the Hardy potential, namely, our equation will be
\begin{equation}    \label{ec:firsteq}
    y_t - \Delta y - \lmb\frac{y}{\abs{x}^2} - a(x)y - v\cdot \nabla y = B_1w(t) + B_2u(t) \;, \;\;\;  \text{in}\;\;(0,\infty)\times \Omega,
\end{equation}
where $\Omega$ is a smooth domain in $\R^N$, $N\geq 3$, $v$ is a convection vector, $w$ is a perturbation and $u$ is the control. More details about the functions and operators in \eqref{ec:firsteq} will be given in the next sections.

Singular inverse-square potentials are not only mathematically intriguing but also play an important role in modeling phenomena across several fields. In quantum mechanics, they appear in relation to inverse-square interactions and uncertainty principles, e.g., \cite{baras, fefferman}. In mathematical physics, they are linked to spectral theory and the asymptotic behavior of heat kernels (\cite{davies}), as well as in combustion theory (\cite{bebernes, gelfand}). From the perspective of PDEs and control theory, the handling of problems with the presence of Hardy potentials requires refined analytical tools due to the lack of standard regularity and the appearance of critical parameters. The literature addressing these issues includes notable works on observability and controllability for equations with singular coefficients, e.g., \cite{cazacuSIAM, ervedoza, vancost-zuazua}.

The present work also deals with the classical Hardy inequality, which plays a central role in our analysis. This inequality states that for any domain $\Omega \subset \mathbb{R}^N$, $N \geq 3$, with $0 \in \Omega$, it holds that
\begin{equation} \label{ec:HardyClasic}
\int_\Omega |\nabla u|^2 dx \geq \lambda \int_\Omega \frac{u^2}{|x|^2} dx, \quad \forall u \in C_c^\infty(\Omega),
\end{equation}
where $\lmb\leq H_N := \left( \frac{N-2}{2} \right)^2$. It is important to mention that $H_N$ is the optimal constant in the inequality \eqref{ec:HardyClasic}, in the sense that the inequality does not hold with a bigger constant. 
The theory surrounding Hardy inequalities has grown significantly, with extensions and generalizations documented in works such as \cite{azorero, balinsky, brevaz, hardy2, kufner}. 

We will address two situations related to the $H^\infty-$control problem, presented in Section 2, associated to \eqref{ec:firsteq}: the subcritical case for $\lmb< H_N$, discussed in Section 3, and the more challenging case for the critical parameter $\lmb=H_N$ in Section 4, treated in a special functional framework.

\section{$H^\infty$-control problem presentation and preliminaries}

\quad In this section, we give a short description of how $H^\infty$-control method works and we list some preliminaries. Consider the system
\begin{equation}  \label{ec:E}
\left\{ \begin{array}{l l}   
y'(t) =Ay(t) + B_1w(t) + B_2u(t)  ,& \text{for}\;\; t\in (0,\infty),   \\
z(t)=C_1y(t)+D_1u(t)  ,&\text{for} \;\;t\in (0,\infty),    \\
y(0)=y_0     \\    
\end{array} \right.
\end{equation}
where $y$ is the system state, $w$ is the perturbation input of the system, $u$ is the control input, $z$ is the performance output, $y_0$ is the initial state and $A$, $B_1$, $B_2$, $C_1$ and $D_1$ are linear operators defined in relation with the real Hilbert spaces $H$, $U$, $W$ and $Z$ identified with their duals which will be well precised later. We assume $A$ is a linear and continuous operator $A\in L(D(A),H)$ with the domain $D(A)=\{y\in H \;|\; Ay\in H \}$ dense in $H$. 
We denote by $A^*$ the adjoint of $A$, with domain $D(A^*)$, which is a closed subspace of $H$, so that it becomes a Hilbert space endowed with a scalar product induced by the norm 
\begin{equation}
    \|y\|_{D(A^*)}:= \left( \|A^* y\|_H^2 + \|y\|_H^2 \right)^\frac{1}{2}, \quad \forall y\in D(A^*).    
\end{equation}
We denote by $(D(A^*))'$ the dual space of $D(A^*)$, which is the completion of $H$ in the norm $\|y\|=\||(A-\lmb_0 I)^{-1}y\|_H$, for $\lmb_0$ regular value of $A$, i.e. $A-\lambda_0 I$ is invertible.

The following definitions will be needed in what follows. 
\begin{definition}   \label{defExpStab}
    An operator $\mathcal{A}$ generates an exponentially stable semigroup $e^{\mathcal{A}t}$ if there exist positive constants $\alpha$ and $C$ such that 
    \begin{equation}
        \|e^{\mathcal{A}t}y\|_H \leq C e^{-\alpha t}\|y\|_H, \;\; \forall y\in H, \;\; \forall t\geq 0.
    \end{equation}
\end{definition} 
According to \cite{datko}, this property is equivalent to
\begin{equation}  \label{ec:datko}
    \int_0^\infty \|y(t)\|_H^2 dt \;<\; \infty, \;\;\forall y\in H.
\end{equation}
\begin{definition}
    Let $\mathcal{A}$ and $\mathcal{C}$ be linear operators on the Hilbert spaces $X$ and $Y$ respectively. The pair $(\mathcal{A},\mathcal{C})$ is called exponentially detectable if there exists an operator $K\in L(Y,X)$ such that $\mathcal{A}+K\mathcal{C}$ generates an exponentially stable $C_0$-semigroup.
\end{definition}

We can now turn our attention to the initial system \eqref{ec:E} and give some insights about the $H^\infty$-control problem. Assume that, under certain conditions, the system \eqref{ec:E} has a mild solution $y\in C([0,T],H)$ for all $T>0$, any control $u\in L^2(0,T;U)$ and any perturbation $w\in L^2(0,T;W)$. Also, assume that $u$ is chosen in feedback form, meaning there exists a linear, closed and densely defined operator $F: H\rightarrow U$ such that $u=Fy$. This translates to system \eqref{ec:E} becoming
\[   
\left\{ \begin{array}{l l}   
y'(t) = (A+B_2F) y(t) + B_1w(t)  ,& \text{for}\;\; t\in (0,\infty),   \\
z(t) = (C_1+D_1F) y(t)  ,&\text{for} \;\;t\in (0,\infty),    \\
y(0) = y_0  .     \\    
\end{array} \right.
\]

Consequently, the solution of the system \eqref{ec:E}, seen as a pair $(y(t),z(t))$, is only depending on the perturbation $w$. According to \cite[Definition 2.3, pg. 106]{pazy}, the mild solution for the system is 
\begin{align}   \notag
    y(t) & = e^{(A+B_2F)t}y_0 + \int_0^t e^{(A+B_2F)(t-s)}B_1w(s) ds, \\
    z(t) & = (C_1+D_1F)e^{(A+B_2F)t}y_0 + G_Fw(t)  \notag
\end{align}
where $$G_Fw(t):=(C_1+D_1F) \int_0^t e^{(A+B_2F)(t-s)}B_1w(s) ds.$$ This is a linear operator from $L^2(\R_+;W)$ to $L^2(\R_+;Z)$ and represents the transfer of effects of perturbation $w$ to the output $z$ of the system and satisfies $G_Fw(t)=z(t)$ for $y_0=0$. 

The aim of $H^\infty-$control problem is to find a suitable feedback controller $F$ which stabilizes exponentially the system in the sense of Definition \ref{defExpStab} and satisfies a prescribed performance of $G_Fw$ in terms of a given constant $\gamma$, i.e., given $\gamma>0$ then $\|G_F\|_{L\left(L^2(\R_+;W),L^2(\R_+;Z)\right)}<\gamma$ in the operator norm. For correlations of this technique with the Hardy space $H^\infty$ one can check \cite{curtain, marinoschi}. 

Our goal is to give an application of the general result in \cite{marinoschi}, which solves the $H^\infty-$control problem for parabolic systems, by considering a singular operator with a convective term. For the sake of completeness, we recall this result, based on the following additional hypotheses for system \eqref{ec:E}:\\
\vspace{0.3cm}
$(i_1)$ $A$ is the infinitesimal generator of an analytic $C_0$-semigroup $e^{At}$ on the Hilbert space $H$ and $e^{At}$ is compact for $t>0$. Moreover, 
\begin{equation}   \label{ec:ipi1}
    B_1\in L(W,H),\; B_2\in L(U,D(A^*)'),\; C_1\in L(H,Z),\; D_1\in L(U,Z).
\end{equation}
$(i_2)$ The pair $(A,C_1)$ is exponentially detectable and 
\begin{equation}    \label{ec:ipi3}
    \int_0^\infty \|B_2^*e^{(A^*+C_1^*K^*)t}y\|_U dt \leq C\|y\|_H, \;\; \forall y\in H. 
\end{equation}\\
$(i_3)$ The following relations hold
\begin{equation}   \label{ec:ipi4}
    \|D_1^*D_1u\|_{U^*}=\|u\|_U \text{\;\;and\;\;} D_1^*C_1=0.
\end{equation}

Now we can state the aforementioned theorem.
\begin{theorem} \cite[Theorem 3.1]{marinoschi}   \label{thmGM}
    Assume that $(i_1)-(i_3)$ hold and let $\gamma>0$. Suppose there exists an operator $F\in L(H,U)$ such that $A+B_2F$ generates an exponentially stable $C_0$-semigroup on $H$ and that
    \begin{equation}    \label{ec:hinfTh}
        \|G_F\|_{L\left(L^2(\R_+;W),L^2(\R_+;Z)\right)} <\gamma.
    \end{equation}
    Then, there exists a Hilbert space $\mathcal{X}\subset H$ with dense and continuous injection and an operator 
    \begin{equation}   \label{ec:propTh}
        P\in L(H,H)\cap L(\mathcal{X},D(A^*)) \;\;\text{with}\;\; P=P^*\geq 0
    \end{equation}
    which satisfies the following algebraic Ricatti equation
    \begin{equation}    \label{ec:ricatti}
        A^*Py + P(A-B_2B_2^*P + \gamma^{-2}B_1B_1^*P)y + C_1^*C_1y = 0, \;\; \forall y\in\mathcal{X}.
    \end{equation}
    Here, $B_2^*P\in L(\mathcal{X},U)$ and the operators 
    \begin{equation}
        \Lambda_P:= A-B_2B_2^*P + \gamma^{-2}B_1B_1^*P \;\;\text{and}\;\; \Lambda_P^1:=A-B_2B_2^*P,
    \end{equation}
    with the domain $\mathcal{X}$, generate exponentially stable semigroups on $H$. Moreover, the feedback control $\tilde{F}:=-B_2^*P$ solves the $H^\infty-$problem, that is 
    \begin{equation}    \label{ec:G_F-norm}
        \|G_{\tilde{F}}\|_{L\left(L^2(\R_+;W),L^2(\R_+;Z)\right)} <\gamma.
    \end{equation}
    Conversely, assume that there exists a solution $P$ to equation \eqref{ec:ricatti} with the properties \eqref{ec:propTh} and such that the corresponding operators $\Lambda_P$ and $\Lambda_P^1$ generate exponentially stable semigroups on $H$. Then, the feedback operator $\tilde{F}:=-B_2^*P$ solves the $H^\infty$-problem \eqref{ec:hinfTh}.
\end{theorem}
In the next sections, we denote by $\mathcal{D}(\Omega )$ the
space $C_{c}^{\infty }(\Omega )$ of infinitely differentiable functions with
compact support in $\Omega ,$ equipped with the strict inductive limit
topology and by $\mathcal{D}^{\prime }(\Omega )$ its strongly dual, that is
the space of all linear functionals defined on $C_{c}^{\infty }(\Omega )$
with values in $\mathbb{R}$. The pairing between two dual spaces $X$ and $%
X^{\prime }$ is denoted by $\left\langle \cdot ,\cdot \right\rangle
_{X^{\prime },X}$ and the norm in a normed space $X$ will be indicated by $%
\left\Vert \cdot \right\Vert _{X}.$ For simplicity, we shall use in some
places the notations: $H:=L^{2}(\Omega ),$ $V:=H_{0}^{1}(\Omega ),$ $%
V^{\prime }=H^{-1}(\Omega ),$ where $V\subset H\subset V^{\prime }$ with
continuous, dense and compact injections. In particular, the scalar product
and the norm in $L^{2}(\Omega )$ will be denoted by $\left( \cdot ,\cdot
\right) _{2}$ and $\left\Vert \cdot \right\Vert _{2},$ respectively. We denote by $\textbf{n}$ the outward unit normal vector to $\prt\Om$.

\section{The subcritical case}

\quad We are interested in solving the $H^\infty-$control problem for the following system
\begin{equation}   \label{ec:Esingular}
\left\{ \begin{array}{l l}   
y_t - \Delta y -\lmb \frac{y}{|x|^2} - a(x)y - v\cdot \nabla y = B_1w(t) + B_2u(t)  ,& \text{in}\;\; (0,\infty)\times\Om,\\
y=0  ,&\text{on}\;\; (0,\infty)\times\Gamma, \\
y(0)=y_0  ,&\text{in} \;\;\Om,     \\    
z(t)=C_1y+D_1u(t)  ,&\text{in} \;\;  (0,\infty)\times\Om,
\end{array} \right.
\end{equation}
where $\Om$ is an open bounded subset of $\R^N$, $N\geq 3$, with $0\in\Om$ and sufficiently smooth boundary $\Gamma=\prt\Om$, the parameter $\lmb$ satisfies 
\begin{equation}
    0<\lmb< H_N,
\end{equation}
$a$ is a function defined as 
\begin{equation}    \label{ec:ec00}
a(x)=a_0\chi_{\Om_0},\;\; a_0>0,\;\; \Om_0\subset\Om,
\end{equation}
$\chi_{\Om_0}$ is the characteristic function of $\Om_0$ and 
\begin{equation}      \label{ec:vCond}
v\in (L^\infty(\Om))^N, \;\;\text{with the property that}\;\;\; \nabla \cdot v\in L^\infty(\Om).
\end{equation} 
Let $y_0\in L^2(\Om)$. We consider
\begin{equation}
    H = W = Z = L^2(\Om), \;\;\; U=\R.     \label{ec:ec01}
\end{equation}
We define the operators 
\begin{align} 
    B_1:&L^2(\Omega)\to L^2(\Omega), \;\; B_1w\;=\; \chi_{\omega_1}(\cdot) w,    \label{ec:ec02}\\
    B_2:&\R\to L^2(\Omega), \;\; B_2 u\;=\;b(\cdot) u,         \label{ec:ec03}\\
    C_1:&L^2(\Omega)\to L^2(\Omega), \;\; C_1y\;=\; \chi_{\Om_C}(\cdot)y,    \label{ec:ec04}\\
    D_1:&\R\to L^2(\Omega), \;\; D_1 u\;=\;\chi_{\Om\setminus\Om_C}(\cdot)u,   \label{ec:ec05}
\end{align}
with conditions 
\begin{equation}
b\in L^2(\Om) \quad \text{and} \quad\omega_1\subset \subset \Om,\;\; \Om_0\subset\subset\Om_C\subset\Om.    \label{ec:ec06}
\end{equation}
We also consider the operator  $A:D(A)\subset L^2(\Om)\rightarrow L^2(\Om)$,
\begin{equation}    \label{ec:defA}
    \langle Ay, z \rangle_{V',V} := \int_\Omega \left\{ -\nabla y \cdot \nabla z + \lambda \frac{y z}{|x|^2} + a(x) y  z + (v \cdot \nabla y) z \right\} dx, \;\;\text{for all} \;\; z\in H_0^1(\Om),
\end{equation}
with the domain 
\begin{align}
    D(A) & =\{y\in H_{0}^{1}(\Omega );\text{ }Ay\in L^{2}(\Omega )\}     \notag\\
    & = \{y\in H_0^1(\Om);\; \exists\; C_y>0 \text{ such that } \left\vert \langle Ay,z \rangle_{V',V} \right\vert \leq C_y \|z\|_2,\; \text{for all}\; z\in H_0^1(\Omega)\}.
\end{align}  

Thus, the equation from the system can be rewritten as
\begin{equation}  \label{ec:ec1}
y'(t) = A y(t) + B_1(x) w(t) + B_2(x) u(t), \;\;\;t > 0.
\end{equation}

Our main result in this case is the following.
\begin{theorem}   \label{thm1}
    Let $\gamma>0$ and $0<\lmb< H_N$. Consider $A,\; B_1,\; B_2,\; C_1$ and $D_1$ defined by \eqref{ec:defA} and \eqref{ec:ec02}-\eqref{ec:ec05}, respectively, and assume \eqref{ec:vCond}. Then, there exists $\tilde{F}\in L(L^2(\Om),\R)$, which solves the $H^\infty$-control problem for system \eqref{ec:Esingular}, given in feedback form by
    $$\tilde{F}=-B^*_2P,$$
    where the operator $P$ satisfies the algebraic Ricatti equation \eqref{ec:ricatti}.
\end{theorem}
First, notice that, due to $0<\lmb<H_N$, the inequality \eqref{ec:HardyClasic} gives us the following estimate 
\begin{equation}  \label{ec:Hardy-consq}
    \int_\Om \left( |\nabla y|^2 - \lmb \frac{|y|^2}{|x|^2} \right) dx \geq C_N \int_{\Om} |\nabla y|^2 dx, \quad C_N:= 1-\frac{\lmb}{H_N}>0,\;\;\;\forall y \in H_0^1(\Om),
\end{equation}
To prove Theorem \ref{thm1}, we have to check the hypotheses $(i_1)-(i_3)$. 

$(i_1)$ By \eqref{ec:ec01}-\eqref{ec:ec05}, it holds that
\begin{equation}
    B_1, C_1\in L(L^2(\Om),L^2(\Om)),\; B_2, D_1\in L(\R,L^2(\Om)).  \label{ec:operatori1}
\end{equation}
\begin{lemma}    \label{lemmaAaccr-subcrt}
For $\lmb<H_N$, the operator $-A$ is $\omega$-m-accretive and $A$ generates a compact $C_0$-semigroup on $L^2(\Om)$, provided that $\omega > \omega_0:=a_0 + \frac{\|\nabla\cdot v\|_\infty}{2}$.
\end{lemma}
\begin{proof}
First step is to show that $-A$ is $\omega$-accretive, meaning that for some $\omega>0$  and all $y\in D(A)$ it holds that $\langle(\omega I-A)y,y\rangle_{V',V}=((\omega I - A)y,y)_2\geq 0$. Taking into account \eqref{ec:Hardy-consq} and $y=0$ on $\prt\Om$, by integration by parts we get
\begin{align}
    ((\omega I-A)y,y)_2 & = \int_{\Om} \left( |\nabla y|^2- \lmb \frac{|y|^2}{|x|^2} \right)  dx + \int_{\Om} (\omega-a) |y|^2  dx - \int_\Om \left( v\cdot \nabla y \right) y dx    
 \notag\\
    & \geq C_N \|\nabla y\|_2^2 + (\omega-a_0) \int_\Om |y|^2 dx - \frac{1}{2} \int_{\prt\Om} |y|^2 (v\cdot \textbf{n}) dx + \frac{1}{2} \int_\Om |y|^2 \nabla\cdot v dx \notag\\
    & \geq C_N \|\nabla y\|_2^2 + \left[\omega - \left( a_0 + \frac{\|\nabla\cdot v\|_\infty}{2} \right)  \right] \|y\|_2^2, \;\; \text{for all}\;y\in D(A).   \label{ec:accretive}   
\end{align}
This shows that $-A$ is $\omega$-accretive for $\omega > \omega_0$.

The second step is to prove that $\omega I - A$ is surjective, i.e. the range of $\omega I-A$ is the entire $L^2(\Om)$ space. Let $f\in L^2(\Om)$. Consider the following bilinear form, defined for every $y,z\in H_0^1(\Om)$: 
\begin{equation}
    \alpha(y, z):=\omega (y, z)_{L^2} - \langle Ay,z\rangle_{V',V}.
\end{equation}
We prove that the equation $\omega y - Ay=f$ has a unique weak solution $y\in D(A)$, namely, there exists $y\in H_0^1(\Omega)$ such that 
\begin{equation}\label{weaksol}
\alpha(y, z)=(f, z)_{L^2}, \quad \forall z\in H_0^1(\Omega).
\end{equation}
This is a consequence of the Lax-Milgram lemma. Indeed, $\alpha$ is a continuous bilinear form on $H_0^1(\Om)$, which is also coercive due to \eqref{ec:accretive} for $\omega > \omega_0 $. We conclude that $\omega I - A$ is surjective. Hence, $-A$ is quasi $m$-accretive and then it generates a $C_0$-semigroup.

It remains to show that $e^{At}$ is compact, that is $(\omega I-A)^{-1}:L^2(\Omega)\mapsto L^2(\Omega)$ is a compact operator for $\omega > \omega_0$. For any $f\in L^2(\Om)$, taking $z=y$ in \eqref{weaksol}  and estimating as in \eqref{ec:accretive}, we obtain that
\begin{equation}
     (\omega - \omega_0) \|y\|_2^2 + C_N \|\nabla y\|_2^2 \leq \|f\|_2 \| y\|_{2}\leq C_P\|f\|_2 \| \nabla y\|_{2},  \label{ec:ec4}
\end{equation}
where the last estimate is the Poincaré inequality with some constant (not necessarily sharp) $C_P=C_P(\Omega)>0$.
Therefore, $\|y\|_{H_0^1(\Om)}\leq C\|f\|_2$ for $\omega$ large enough and some constant $C>0$. Let $\{f_n\}_n$ be a bounded sequence in $L^2(\Om)$. Since $\|(\omega I-A)^{-1}f_n\|_{H_0^1(\Om)}\leq C\|f_n\|_2$  we get that $\{ (\omega I-A)^{-1}f_n \}_n$ is bounded in $H_0^1(\Om)$. Since $H_0^1(\Om)$ is compactly embedded in $L^2(\Om)$, we can subtract a subsequence which converges in $L^2(\Om)$. Hence, $(\omega I-A)^{-1}$ is compact.  
\end{proof}

\begin{lemma}   \label{lemmaAanltc-subcrt}
    The operator $A$ generates an analytic $C_0$-semigroup.
\end{lemma}
\begin{proof}
We consider $A=A_0+B$, where
\begin{equation}  \label{eq:Adecomposition}
    A_0y=\Delta y + \lmb\frac{y}{|x|^2}+ay \;\;\;\text{and}\;\;\; By=v\cdot\nabla y   
\end{equation}
are defined in the sense of distributions, with $D(B)=\left\{ y\in L^2(\Om)\;;\; By \in L^2(\Om) \right\}$ and $D(A_0)=\{y\in H_0^1(\Om)\;|\;A_0y\in L^2(\Om) \}$. We note that $D(A_0)\subset D(B)$ and that $A_0$ is the generator of an analytic $C_0$-semigroup (see \cite[Lemma 4.1]{marinoschi}).

To prove the analyticity of $A$, we use the result stated in \cite[Section 3.2, Theorem 2.1, pg. 80]{pazy}. To this end, we need to show that $B$ is closed and there exist constants $c_1, c_2>0$ such that 
\begin{equation}     \label{ec:analytic-condition}
    \|By\|_2\leq c_1\|A_0y\|_2 + c_2\|y\|_2,
\end{equation}
for all $y\in D(A_0)$. 

Take $y_n\rightarrow y$ in $L^2(\Om)$ such that $By_n\rightarrow\eta$ in $L^2(\Om)$. Let $\phi\in C_c^\infty(\Om)$. Then,
\begin{equation}   \notag
    (By_n,\phi)_2= -\int_\Om y_n \nabla\cdot(v\phi) dx \stackrel{n\to\infty}{\longrightarrow} -\int_\Om y \nabla\cdot(v\phi) dx = (By,\phi)_2.   
\end{equation}
On the other hand, 
\begin{equation}
    (B y_n,\phi)_2 \stackrel{n\to\infty}{\longrightarrow} (\eta,\phi)_2,
\end{equation}
hence $By=\eta$, and we conclude that $B$ is closed.   
In order to show \eqref{ec:analytic-condition}, we point out the next estimates, using again \eqref{ec:Hardy-consq}:
\begin{align}   \notag
    (-A_0y,y)_2 & = \int_\Om \left( |\nabla y|^2 - \lmb \frac{|y|^2}{|x|^2} \right) dx - a_0\int_{\Om_0}|y|^2 dx      \notag\\
    & \geq C_N \int_\Om|\nabla y|^2 dx - a_0\int_\Om |y|^2 dx   \notag\\
    & \geq  C_N\|\nabla y\|_2^2 - a_0\|y\|_2^2, \;\; \text{for all}\; y\in D(A_0).
\end{align}
Therefore, 
\begin{align}
    \|\nabla y\|_2^2 & \leq \frac{1}{C_N} (-A_0y,y)_2 + \frac{a_0}{C_N}\|y\|_2^2, \;\;     \notag\\
    & \leq \frac{1}{C_N} \|-A_0y\|_2\|y\|_2 + \frac{a_0}{C_N}\|y\|_2^2.   \label{ec:H01norm}
\end{align}
Using \eqref{ec:H01norm} and Young inequality, we have
\begin{align}
    \|By\|_2^2 &= \|v\cdot \nabla y\|_2^2 \leq \|v\|_\infty^2 \|\nabla y\|_2^2  \leq   \frac{\|v\|_\infty^2}{C_N} \|-A_0y\|_2\|y\|_2 + \frac{\|v\|_\infty^2 a_0}{C_N}\|y\|_2^2    \notag\\
    & \leq \frac{\|v\|^2_\infty}{C_N} \left( \eps \frac{C_N}{\|v\|_\infty^2}\|A_0 y\|_2^2 + \frac{1}{4\eps C_N}\|v\|_\infty^2\|y\|_2^2 \right) + \frac{\|v\|_\infty^2 a_0}{C_N}\|y\|_2^2  \notag\\
    & = \eps \|A_0y\|_2^2 + \frac{\|v\|_\infty^2}{C_N} \Big( \frac{\|v\|_\infty^2}{4\eps C_N} + a_0 \Big)\|y\|_2^2.
\end{align}
This proves the condition \eqref{ec:analytic-condition}. By the theorem mentioned above, we conclude that $A=A_0+B$ is the infinitesimal generator of an analytic semigroup.   
\end{proof}

$(i_2)$ We prove the following lemma.
\begin{lemma}   \label{lemmaAexpstb-subcrt}
    Let $\lmb<H_N$. Then the pair $(A,C_1)$ is exponentially detectable.
\end{lemma}
To prove that $(A,C_1)$ is exponentially detectable translates to finding $K\in L(L^2(\Om),L^2(\Om))$ such that $A+KC_1$ generates an exponentially stable semigroup, that is
\begin{equation}          \label{ec:expdetect}
   \|e^{(A+C_1K)t}y\|_2  \leq C e^{-\alpha t}\|y\|_2, \;\; \forall y\in L^2(\Om),\;\; \forall t\geq 0, 
\end{equation}
where $C$ and $\alpha$ are positive constants. To prove this, we apply Datko's result in \cite{datko}, which states that the above condition is equivalent to
\begin{equation}    \label{ec:datko}
    \int_0^\infty \|y(t)\|^2_2 dt < \infty, \;\; \forall y\in L^2(\Om). 
\end{equation}
\begin{proof}
Consider 
\begin{equation}
    K=-kI, \;\;\; k> \omega_0 .   \notag
\end{equation}
and set $A_1=A+KC_1$. This is still m-accretive, therefore $-A_1$ generates a $C_0$-semigroup $e^{A_1t}$ on $L^2(\Om)$. Hence $y(t)=e^{A_1t}y_0$ satisfies
\begin{equation}
    \frac{dy}{dt}(t)=A_1y(t), \;\; t\geq 0,\;\; y(0)=y_0.   \label{ec:eqA1}
\end{equation}
Multiplying \ref{ec:eqA1} by $y(t)$, integrating on $\Om$ and applying again the Hardy inequality, we get
\begin{equation}   \notag
    \frac{1}{2}\frac{d}{dt}\|y(t)\|_2^2 + C_N\|\nabla y(t)\|_2^2 + k\int_{\Om_C} |y|^2 dx - \int_{\Om_0} a_0|y|^2 dx - \int_{\Om} \left( v\cdot\nabla y \right) y dx \leq 0.
\end{equation}
Recall that $\Om_0\subset\subset\Om_C$. We integrate from $0$ to $t$:
\begin{align}   
    \frac{1}{2}\|y(t)\|_2^2 & + C_N\int_0^t\|\nabla y(s)\|_2^2 ds + (k-a_0)\int_0^t\int_{\Om_0} |y(s)|^2 dx ds - \int_0^t\int_{\Om} \left( v\cdot\nabla y(s) \right) y(s) dx ds \notag\\
    &\leq \frac{1}{2}\|y_0\|_2^2.   \label{ec:calcul-exp-dect}
\end{align}
Since 
\begin{equation}
    \int_\Om  \left( v\cdot\nabla y \right) y dx = \frac{1}{2} \int_\Om v \cdot \nabla(y^2) dx = -\frac{1}{2} \int_\Om (\nabla \cdot v) y^2 dx \leq -\frac{\|\nabla\cdot v\|_\infty}{2} \|y\|_2,    \notag
\end{equation}
we get by \eqref{ec:calcul-exp-dect} that
\begin{equation}    
    \|y(t)\|_2^2 + 2C_N\int_0^t\|\nabla y(s)\|_2^2 ds + 2( k - \omega_0 ) \int_0^t \|y(s)\|_2^2 dxds \leq \|y_0\|_2^2.   \notag
\end{equation}
Taking into account the bound for $k$, we obtain
\begin{equation}   \notag
    \int_0^t \|y(s)\|_2^2 ds \leq \frac{1}{2(k-\omega_0)} \|y_0\|_2^2, \;\; \text{for all } \;t>0.  
\end{equation}
Letting $t\to\infty$, we get the final estimate 
\begin{equation}   \notag
    \int_0^\infty \|y(s)\|_2^2 ds \leq \frac{1}{2(k-\omega_0)} \|y_0\|_2^2 < \infty.  
\end{equation}
Finally, by \cite{datko}, there exists $\alpha>0$ such that 
\begin{equation}   \notag
    \|e^{(A+KC_1)t} y\|_2 \leq e^{-\alpha t} \|y\|_2, \;\; \forall y\in L^2(\Om),
\end{equation}
that is, the pair $(A,C_1)$ is exponentially detectable. Moreover, it is easy to check that
\begin{equation}   \notag
    \int_0^\infty \left| B_2^* e^{(A^*+C_1^*K^*)t}y \right| dt \leq C \int_0^\infty \|e^{(A^*+C_1^*K^*)t}y\|_2 dt \leq C\|y\|_2 \int_0^\infty e^{-\alpha t} dt = C\|y\|_2,
\end{equation}
for all $y\in L^2(\Om)$.
\end{proof}
$(i_3)$ $\|D_1^*D_1u\|_{U^*}=\|u\|_U$ and $D_1^*C_1=0$. 
Considering \eqref{ec:ec01}-\eqref{ec:ec05}, this can be checked by direct computations. \\

The conditions in Theorem \ref{thmGM} being fulfilled, it follows that the $H^\infty$-control problem for system \eqref{ec:Esingular} has a solution, hence Theorem \ref{thm1} holds. We can now write the $H^\infty-$control problem for system \eqref{ec:Esingular}. Precisely, there exists a Hilbert space $\mathcal{X}\subset L^2(\Om)$ and an operator 
\begin{equation}
    P\in L(L^2(\Om)L^2(\Om))\cap L\left(\mathcal{X},D(A^*)\right), \;\; P=P^* \geq 0,
\end{equation}
which satisfies the algebraic equation \eqref{ec:ricatti}, such that $\tilde{F}:=-B_2^*P$ solves the $H^\infty$-problem for the system \eqref{ec:Esingular}, that is relation \eqref{ec:G_F-norm}. In our case, we can see that $\mathcal{X}=D(A)$. 

As an observation, we can give another representation of the Ricatti equation \eqref{ec:ricatti}, which may be more convenient for numerical computations. We recall that the linear continuous operator $P\in L(L^2(\Om), L^2(\Om))$ can be represented by the Schwartz kernel theorem (see, e.g., \cite[p. 166]{Lions}) as an integral operator with a kernel $P_0\in L^2(\Om \times \Om)$, namely 
\begin{equation}   \label{ec:kernelP}
    P\varphi(x) = \int_\Om P_0(x,\xi)\varphi(\xi) d\xi, \;\; \text{for all}\;\; \varphi\in C_0^\infty(\Om).
\end{equation}
We make use of the same computations as in \cite{marinoschi}, by replacing \eqref{ec:kernelP} in \eqref{ec:ricatti}, and obtain 
\begin{align}    \label{ec:ricattieq}
      \Delta_x P_0(x,\xi) &+ \Delta_\xi P_0(x,\xi) + \lmb P_0(x,\xi) \left(\frac{1}{|x|^2
      } + \frac{1}{\abs{\xi}^2} \right) + v(x) \cdot \nabla_x P_0(x,\xi) + v(\xi) \cdot \nabla_\xi P_0(x,\xi)   \notag\\
      & + \left(a(x) + a(\xi)\right) P_0(x,\xi)  
      - \int_\Om\int_\Om P_0(x,\overline{\xi}) P_0(\overline{x},\xi) b(\overline{\xi})b(\overline{x})d\overline{x}d\overline{\xi}     \notag\\
      \gamma^{-2} &\int_\Om \chi_{\omega_1}(\overline{\xi})P_0(x,\overline{\xi}) P_0(\overline{\xi},\xi) d\overline{\xi} = -\delta(x-\xi)\chi_{\Om_C}(\xi), \;\; \text{in}\; \mathcal{D}'(\Om\times\Om),
    \end{align}
    accompanied by the conditions
    \begin{align}     
        P_0(x,\xi)& =0, \;\; \forall(x,\xi)\in \Gamma\times\Gamma,  \label{ec:ricatticonditions1}\\
        P_0(x,\xi)& =P_0(\xi,x), \;\;\forall (x,\xi)\in\Omega\times\Omega,  \label{ec:ricatticonditions2}\\
        P_0(x,\xi)& \geq 0, \;\;\forall (x,\xi)\in\Omega\times\Omega.    \label{ec:ricatticonditions3}
    \end{align}
    Moreover, we get that
    \begin{equation}    \label{ec:controlformula}
        \tilde{F}y=-\int_\Om\int_\Om b(x)P_0(x,\xi) y(\xi) d\xi dx, \;\;\forall y\in L^2(\Om).
    \end{equation}
    We point out the differences in computations for the terms involving the operator $A$:
\begin{align}
    A^*P\varphi(x) &= \int_\Om \Big(\Delta_x P_0(x,\xi) + \frac{\lmb P_0(x,\xi)}{|x|^2} + a(x)P_0(x,\xi) + v(x)\cdot\nabla_x P_0(x,\xi) \Big)\varphi(\xi) d\xi,  \notag\\
    PA\varphi(x) &= \int_\Om \varphi(\xi)\Big(\Delta_\xi P_0(x,\xi) + \frac{\lmb P_0(x,\xi)}{|\xi|^2} + a(\xi)P_0(x,\xi) + v(\xi)\cdot\nabla_\xi P_0(x,\xi) \Big) d\xi.  \notag
\end{align}  
These lead to the equation \eqref{ec:ricattieq} and accompanying conditions \eqref{ec:ricatticonditions1}-\eqref{ec:ricatticonditions3}, which are equivalent to the Ricatti equation \eqref{ec:ricatti}.

\section{The critical case}

\quad Here, we analyze the problem \eqref{ec:Esingular} in the case $\lmb=H_N$. We consider the same framework as before, along with all notations and definitions from \eqref{ec:ec00}-\eqref{ec:ec06}.
We notice that the choice of critical parameter $\lmb=H_N=\frac{(N-2)^2}{4}$ creates a dichotomy with respect to the sub-critical case $\lmb<H_N$, since the Hardy inequality \eqref{ec:HardyClasic} no longer provides convenient estimates for our needs (e.g., it fails to provide the quasi m-accretivity of $A$). Therefore, we need an improved form of the Hardy inequality, namely the one from \cite{VazquezZuazua} recalled below, and so we move forward with some preliminaries. The treatment of the critical case $\lmb=H_N$ is new even concerning \cite{marinoschi}, without the convection term, since the methods used to assert all the hypotheses are slightly different.

\begin{theorem}\cite[Theorem 2.2]{VazquezZuazua}    \label{thmHardyCritic}
Let $\Om$ be an open bounded subset of $\R^N$, $N\geq 3$. Then for any $1\leq p<2$ there exists $C(p,\Om)>0$ such that 
\begin{equation}    \label{ec:HardyCritic}
\int_\Om \left( |\nabla y|^2 - H_N\frac{|y|^2}{\vert x\vert^2} \right) dx \;\geq\; C(p,\Om)\|y\|_{W^{1,p}(\Om)}^2      
\end{equation}
holds for all $y\in H_0^1(\Om)$.
\end{theorem}

As far as we are concerned, the optimal value of the constant $C(p,\Om)$ is not known in general. According to, e.g., \cite{brevaz}, it is known that $C(p,\Om)>0$ and is actually inversely proportional to the diameter of $\Omega$. We define, as in \cite{VazquezZuazua}, a space strictly larger than $H_0^1(\Omega)$:   
\begin{equation}    \label{ec:defHcritic}
    \HH:=\overline{C_c^\infty(\Om)}^{\|\cdot\|_\HH},
\end{equation}
which is a Hilbert space with the norm 
\begin{equation}    \label{H-norm}
     \|y\|_\HH :=\Bigg( \int_\Om \left( |\nabla y|^2 - H_N\frac{|y|^2}{\vert x\vert^2} \right) dx \Bigg)^{\frac{1}{2}},
\end{equation}
associated to the bilinear form  
\begin{equation}
    \Tilde{\alpha}(y,\phi)= \int_\Om \left( \nabla y\cdot \nabla \phi - H_N \frac{y\phi}{|x|^2} \right) dx, \;\;\text{defined for all}\;\; y,\phi\in C^\infty_c(\Om).
\end{equation}
Thus, according to this definition, Theorem \ref{thmHardyCritic} asserts that the following continuous embedding holds:
\begin{equation}    \label{eq:HinW1p}
    \HH\subset W^{1,p}(\Om), \;\;\text{for all } 1\leq p<2. 
\end{equation}
By the Sobolev-Gagliardo-Nirenberg inequality (\cite[Corollary 9.14, pg. 285]{brezis}) and the Rellich-Kondrachov Theorem (\cite[Theorem 9.16, pg. 285]{brezis}), we also have 
\begin{equation}   \label{eq:SobolevGaNir}
    W^{1,p}(\Om)\subset L^{p^*}(\Om), \; \frac{1}{p^*}=\frac{1}{p}-\frac{1}{N},\; p<N,
\end{equation}
with continuous, dense and compact injection. Since for $p^*=2$ we have $p=\frac{2N}{N+2}<2$ for any $N\geq 3$, it also follows that
\begin{equation}   \label{HinL2}
    \HH\subset L^2(\Om)
\end{equation}
with continuous, dense and compact injection. Since $\HH\subset W^{1,p}(\Om)$ by \eqref{ec:HardyCritic}, then the trace of $y$ on $\Gamma$ makes sense. More information about the space $\HH$ can be found in \cite{VazquezZogra}. In the sequel, for the definition of $A$, we will follow \cite[Section 4]{VazquezZuazua}, defining in this case
\[
A:D(A)\subset L^{2}(\Omega )\rightarrow L^{2}(\Omega ),
\]%
with the domain   
\begin{eqnarray*}
D(A) &=&\{y\in \mathcal{H};Ay\in L^{2}(\Omega ),\text{ }y=0\text{ on }
\partial \Omega \} \\
&=&\{y\in \mathcal{H};\; \exists\; C_y>0 \text{ such that } \left\vert \left\langle Ay,\phi \right\rangle
_{\mathcal{D}^{\prime }(\Omega ),\mathcal{D}(\Omega )}\right\vert \leq
C_y\left\Vert \phi \right\Vert _{2},\text{ for }\phi \in C_{c}^{\infty }(\Omega) \},
\end{eqnarray*}
and its expression given by
\[
\left\langle Ay,\phi \right\rangle _{\mathcal{D}^{\prime }(\Omega ),\mathcal{
D}(\Omega )}=\int_{\Omega }\left\{ \left(- \nabla y \nabla \phi + H_{N}\frac{y\phi }{\left\vert x\right\vert ^{2}} \right) + a(x) y\phi + (v \cdot \nabla y)\phi \right\} dx,
\text{ for all }\phi \in C_{c}^{\infty }(\Omega).
\]
We can take $\phi \in C^\infty_c(\Om)$ above, since $C^\infty_c(\Om\setminus\{0\})$ is dense in $C_c^\infty(\Om)$ with respect to the norm in \eqref{H-norm}, due to \cite[Lemma 2.1]{VazquezZogra}. Hence, the equation from the system \eqref{ec:Esingular} can be written as
\begin{equation}  
y'(t) = A y(t) + B_1(x) w(t) + B_2(x) u(t), \;\;\;t > 0.
\end{equation}

Now we can state our main result in the critical case.
\begin{theorem}   \label{thm2}
    Let $\gamma>0$ and $\lmb = H_N$. Under the conditions of Theorem \ref{thm1}, assume in addition that 
    \begin{equation}    \label{ec:vCritic}
        \|v\|_\infty < C_0(p,\Om)    
    \end{equation}
    for some constant $C_0(p,\Om)$, which depends on $p$ and $\Omega$, that will be precised later. Then, there exists $\tilde{F}\in L(L^2(\Om),\R)$, which solves the $H^\infty$-control problem for system \eqref{ec:Esingular}, given in feedback form by
    $$\tilde{F}=-B^*_2P,$$
    where the operator $P$ satisfies conditions \eqref{ec:propTh} and the algebraic Ricatti equation \eqref{ec:ricatti}.
\end{theorem}

The proof relies on checking the conditions $(i_1)-(i_3)$ from Theorem \ref{thmGM}. We will focus on the parts of the proof that are different than in the previous case (mostly the ones depending on $A$).

$(i_1)$ The condition follows from the following two lemmas. 
\begin{lemma}   \label{lemmaAaccr-crt}
    The operator $-A$ is $\omega$-m-accretive and $A$ generates a compact $C_0$-semigroup on $L^2(\Om)$, provided that $\omega > \omega_0$.
\end{lemma}
\begin{proof}
This means to show that $-A$ is $\omega$-accretive, that is $((\omega I - A)y,y)_2\geq 0$ for some $\omega>0$  and all $y\in D(A)$ and that $\omega I - A$ is surjective. 
First, we prove the result for an approximation of $A$. We define, for any $%
\eps >0$ the operator 
$$ A_{\eps }y=\Delta y+H_{N}\frac{y}{\left\vert x\right\vert ^{2} + \eps}+a(x)y + v \cdot \nabla y,$$
by
$$\langle A_\eps y,\phi \rangle_{V',V} = \int_\Om \left( -\nabla y\cdot\nabla \phi + \lambda \frac{y \phi}{|x|^2+\eps} + a(x) y  \phi + (v \cdot \nabla y) \phi \right) dx, \;\;\text{for all} \;\; \phi\in H_0^1(\Om),$$
with the domain $D(A_\eps)=H^2(\Om)\cap H_0^1(\Om)$, due to the standard elliptic regularity. Let $\lmb_{\eps,N}$ be such that 
$$H_N > \lmb_{\eps,N}\geq H_N \max_{x\in\Om}\left(\frac{|x|^2}{|x|^2+\eps}\right)=H_N \frac{R^2}{R^2+\eps},$$
where $R=\max_{x\in\Om}|x|$. Taking into account the estimate \eqref{ec:Hardy-consq} and that $y=0$ on $\partial \Omega $, we get 
\begin{align}
    ((\omega I-A_\eps)y,y)_2 & = \omega\int_{\Om} |y|^2dx + \int_\Om \left( |\nabla y|^2 - H_N\frac{|y|^2}{|x|^2+\eps} \right) dx - \int_{\Om} a |y|^2  dx - \frac{1}{2}\int_\Om |y|^2 \nabla \cdot v dx
 \notag\\
    & = \int_\Om \left( |\nabla y|^2 - \lmb_{\eps,N}\frac{|y|^2}{|x|^2} \right) dx +\int_\Om \left( \frac{\lmb_{\eps,N}}{|x|^2} - \frac{H_N}{|x|^2+\eps} \right) |y|^2 dx + (\omega - \omega_0) \|y\|_2^2    \notag\\
    & \geq C_{N,\eps} \|y\|_{H_0^1(\Om)}^2 + (\omega - \omega_0) \|y\|_2^2,  \;\text{ for all }\; y\in D(A_\eps) \;\text{and all}\; \eps>0,   \label{ec:accretive-crt}   
\end{align}
where $C_{N,\eps}:=1-\frac{\lmb_{\eps,N}}{H_N}$. This shows that $-A_{\eps }$ is $\omega $-accretive for $\omega >\omega _{0}$. Moreover, it is easily seen that $\omega I-A_{\eps }$ is linear continuous from $H_{0}^{1}(\Omega )$ to $H^{-1}(\Omega).$ Then, using \eqref{ec:accretive-crt} and the Lax-Milgram theorem it follows that $\omega I-A_{\eps }$ is surjective, that is $R(\omega I-A_{\eps })=V^{\prime }.$ In particular, for each $f\in L^{2}(\Omega )$, the equation 
\begin{equation}
\omega y_{\eps }-A_{\eps }y_{\eps }=f  \label{200}
\end{equation}
has a unique solution $y_{\eps }\in H^2(\Om)\cap H_{0}^{1}(\Omega ).$ By (\ref{200}) and since
$$\left( \omega - \omega_0\right) \left\Vert y_{\eps }\right\Vert _{2}^{2} + \|y_\eps\|_\HH^2 \leq ((\omega I_\eps-A_\eps) y_\eps,y_\eps)_2,$$
we deduce that $y_\eps$ satisfies the estimate
\begin{equation}
\left( \omega - \omega_0\right) \left\Vert y_{\eps }\right\Vert _{2}^{2}+\left\Vert y_{\eps }\right\Vert _{\mathcal{H}}^{2} \leq  \left\Vert f\right\Vert_{2}\left\Vert y_{\eps}\right\Vert _{2},  \label{201}
\end{equation}
implying 
\begin{equation}
\left( \omega -\omega _{0}\right) \left\Vert y_{\eps }\right\Vert
_{2}\leq \left\Vert f\right\Vert _{2}  \label{202}
\end{equation}
and also, using the last inequality,
\begin{equation}
\left\Vert y_{\eps }\right\Vert _{\mathcal{H}}^{2}\leq \frac{1}{%
\omega -\omega _{0}}\left\Vert f\right\Vert _{2}^{2}.  \label{203}
\end{equation}
Thus, for $\omega >\omega _{0},$ we deduce that $\{y_{\eps
}\}_{\eps }$ is bounded in $\mathcal{H}.$ Since $\mathcal{H}$ is
compactly embedded in $L^{2}(\Omega )$ it follows that on a subsequence 
$$ y_{\eps }\rightarrow y\text{ weakly in }\mathcal{H}\text{ and
strongly in }L^{2}(\Omega )\text{, as }\eps \rightarrow 0. $$
By (\ref{200}) it follows that 
$$ A_{\eps }y_{\eps }\rightarrow f-\omega y\text{ strongly in }%
L^{2}(\Omega )\text{, as }\eps \rightarrow 0. $$
Taking into account that 
\begin{align}
(A_{\eps }y_{\eps },\phi )_{2} & = \int_{\Omega }y_{\eps
}\left( \Delta \phi +H_{N}\frac{\phi }{\eps + \left\vert x\right\vert ^{2}}+a\phi
-\nabla \cdot (v \phi )\right) dx   \\
&\int_{\Omega }y_{\eps}\left( \Delta \phi +H_{N}\frac{\phi }{\left\vert x\right\vert ^{2}}+a\phi
-\nabla \cdot (v \phi )\right) dx -\eps \int_{\Omega }H_{N}\frac{\phi y_{\eps }}{\left\vert
x\right\vert ^{2}(\eps +\left\vert x\right\vert ^{2})}dx \\
&\rightarrow (Ay,\phi )_2,\text{ for all }\phi \in C_{c}^{\infty }(\Omega
\backslash \{0\}), \; \text{as }\;\eps\to 0,
\end{align}
we get that 
$$ (f-\omega y,\phi )_{2}=(Ay,\phi )_{2},\text{ for all }\phi \in C_{c}^{\infty
}(\Omega \backslash \{0\}). $$
Hence, by density, 
\[
(f-\omega y,\phi )_{2}=(Ay,\phi )_{2},\text{ for all }\phi \in H_0^1(\Om),\text{ }
\]%
which shows that $\omega y-Ay=f.$
Passing to the limit in \eqref{ec:accretive-crt} written for $y=y_{\eps },$ 
$$\left(\left( \omega I-A \right) y_{\eps } , y_{\eps }\right)
_{2}\geq (\omega -\omega _{0})\left\Vert y_{\eps }\right\Vert _{2}^{2}, $$
we get 
$$(\left( \omega I-A)y,y\right) _{2}\geq (\omega
-\omega _{0})\left\Vert y\right\Vert _{2}^{2},$$
meaning that $\omega I-A$ is accretive. This shows that $-A$ is quasi $m$-accretive on $L^{2}(\Omega )$
and so $A$ generates a $C_{0}$-semigroup. 
Relations (\ref{202}) and (\ref{203}) are preserved by passing to the limit by $\eps\to 0$ and since $\mathcal{H}$ is compactly embedded in $L^{2}(\Omega )$ it follows that the resolvent is compact. Thus the semigroup $e^{At}$ generated by $A$ is compact, for all $t>0.$

\end{proof}

\begin{lemma}   \label{lemmaAanltc-crt}
    The operator $A$ generates an analytic semigroup, provided that \eqref{ec:vCritic} holds.
\end{lemma}
\begin{proof}
We have to prove, according to \cite[pg. 61, Theorem 5.2]{pazy} that there exists $\sigma_0\in\R_+$ such that 
    \begin{equation}   \label{analiticity}
        \|(\sigma I -A)^{-1}f\|_2 \leq \frac{M_f}{\abs{\sigma-\sigma_0}},    
    \end{equation}
for 
\begin{equation}
    \sigma\in\mathbb{C},\;\; \sigma=\sigma_1+i\sigma_2 \;\;\text{with}\;\; \sigma_1>\sigma_0.
\end{equation}
Let $f\in L^2(\Om)$. We consider the equation 
$$(\sigma I-A)y = f,$$
which can be written on components as
    \begin{equation}
        (\sigma_1+i\sigma_2)(y_1+iy_2)-A(y_1+iy_2) = f_1+if_2.   \notag
    \end{equation}
This implies that
\begin{align}   \label{eq:sigmasystem}
    \sigma_1 y_1 - \sigma_2 y_2 - Ay_1 &= f_1,  \notag\\
    \sigma_1 y_2 + \sigma_2 y_1 - Ay_2 &= f_2.  
\end{align}
We prove that system \eqref{eq:sigmasystem} has a solution $y\in D(A)$. Let $\LL:=L^2(\Om)\times L^2(\Om)$ and define on this space the operators
\begin{equation}
    \A \begin{pmatrix}
        \;y_1\; \\ \;y_2\;
    \end{pmatrix}
    = \begin{pmatrix}
        \;-Ay_1\; \\ \;-Ay_2\;
    \end{pmatrix}, \;\;
    \A_0 \begin{pmatrix}
        \;y_1\; \\ \;y_2\; 
    \end{pmatrix}
    = \begin{pmatrix}
        \;\sigma_1 y_1 - \sigma_2 y_2\; \\ \;\sigma_1 y_2 + \sigma_2 y_1\;  
    \end{pmatrix}, \;\; \forall (y_1,y_2)\in \LL.    \notag
\end{equation}
Then, for some $\gamma>0$, we have
\begin{align}
    \left(  \;\A \begin{pmatrix}
        \;y_1\; \\ \;y_2\;
    \end{pmatrix},
    \begin{pmatrix}
        \;y_1\; \\ \;y_2\;
    \end{pmatrix}\; \right) & \;= \; \left(-Ay_1, y_1\right)_2 \;+\; \left(-Ay_2,y_2\right)_2 \; \notag\\
    & \geq \|y\|_{\HH\times\HH}^2 - \omega_0 \|y\|_\LL^2 \notag\\
    & \geq -\gamma \|y\|_\LL^2     \notag
\end{align}
and 
\begin{equation}
    \left(  \;\A_0 \begin{pmatrix}
        \;y_1\; \\ \;y_2\;
    \end{pmatrix}\;,\;
    \begin{pmatrix}
        \;y_1\; \\ \;y_2\;
    \end{pmatrix}\;
    \right) \geq \sigma_1 \|y\|_\LL^2.   \notag
\end{equation}
We conclude that the operator $\A+\gamma I$ is quasi $m-$accretive for $\sigma_1>\gamma$ and $\A_0$ is $m-$accretive and continuous. Then the operator $\A+\A_0$ is quasi-$m-$accretive (cf. \cite[pg. 44, Corollary 2.6]{Barbu1}). Being also coercive for $\sigma_1>\gamma$, it is surjective (cf. \cite[pg. 36, Corollary 2.2]{Barbu1}), proving that $R(\A+\A_0)=\LL$. Therefore, for $\sigma_1>\gamma$ and $\sigma_2\in\R$, system \eqref{eq:sigmasystem} has a solution. \\
We multiply scalarly in $L^2(\Om)$ the first equation in \eqref{eq:sigmasystem} by $y_1$ and the second one with $y_2$. Summing the two equations, by \eqref{ec:accretive-crt}, we get
\begin{align}
    \sigma_1 & \|y_1\|_2^2 + \|y_1\|_\HH^2 - \omega_0 \|y_1\|_2^2   \notag\\
    + \sigma_1 & \|y_2\|_2^2 + \|y_2\|_\HH^2 - \omega_0 \|y_2\|_2^2   \notag\\
    \leq & \|f_1\|_2\|y_1\|_2 + \|f_2\|\|y_2\|_2.   \notag
\end{align}
Hence, we obtain that 
\begin{equation}   \label{eq:sigma1estimate}
    (\sigma_1 - \omega_0) \|y\|_2^2 + \|y\|_\HH^2 \leq \|f\|_2\|y\|_2.
\end{equation}

Now, assume first that $\sigma_2>0$ and multiply scalarly in $L^2(\Om)$ the first equation in \eqref{eq:sigmasystem} by $-y_2$ and the second by $y_1$. Summing up, by the definitions of $A_0$ and $B$ from \eqref{eq:Adecomposition}, we have
\begin{align}  \label{eq:sigma2estimates}
    \|f\|_2\|y\|_2 & \geq \sigma_2\|y\|_2^2 + \left(A_0y_1,y_2\right)_2 + \left(By_1,y_2\right)_2 - \left(A_0y_2,y_1\right)_2 - \left(By_2,y_1\right)_2 \notag\\
    & = \sigma_2\|y\|_2^2 + \left(By_1,y_2\right)_2 - \left(By_2,y_1\right)_2.
\end{align}
The terms with $A_0$ reduce, since $A_0$ is self-adjoint. Also, we have that
\begin{align}   \label{eq:Bestimates}
    (By_1, y_2)_2 - (By_2, y_1)_2 & = \int_\Om y_2 v \cdot \nabla y_1 dx - \int_\Om y_1 v \cdot \nabla y_2 dx   \notag\\
    & = -2 \int_\Om y_1 v \cdot \nabla y_2 dx - \int_\Om y_1 y_2 \nabla\cdot v dx  \notag\\
    & \leq 2\|v\|_\infty\|y_1\|_{p'}\|\nabla y_2\|_p + \|\nabla\cdot v\|_\infty\|y_1\|_2\|y_2\|_2, 
\end{align}
where we note by \eqref{ec:HardyCritic} that $\nabla y_2\in \big(L^p(\Omega)\big)^N$ and $\frac{1}{p'} = 1 - \frac{1}{p}$. We choose $p$ such that $\frac{2N}{N+1}\leq p<2$, hence $p^*>p'$. Indeed, we have
\begin{equation}
    \frac{1}{p^*} - \frac{1}{p'} = \frac{1}{p}-\frac{1}{N}-1+\frac{1}{p} = \frac{2}{p}-\frac{N+1}{N}<0.  \notag
\end{equation}
Thus, by \eqref{eq:SobolevGaNir}, it holds that
\begin{equation}
    W^{1,p}(\Om) \subset L^{p^*}(\Om) \subset L^{p'}(\Om),
\end{equation}
and so there exists $C_{p,\Om}>0$ such that $\|y_1\|_{p'} \leq C_{p,\Om} \|y_1\|_{W^{1,p}(\Om)}$. Finally, by using \eqref{eq:Bestimates} and \eqref{eq:HinW1p} in \eqref{eq:sigma2estimates}, we get
\begin{equation}  \label{eq:sigma2estimates2}
    (\sigma_2 - \|\nabla\cdot v\|_\infty) \|y\|_2^2 \leq \|f\|_2\|y\|_2 + 2C_{p,\Om}\|v\|_\infty \|y_1\|_{W^{1,p}(\Om)}\|y_2\|_{W^{1,p}(\Om)}.
\end{equation}
Now we add \eqref{eq:sigma1estimate} and \eqref{eq:sigma2estimates2} and use \eqref{ec:HardyCritic}:
\begin{align}
    (\sigma_1-\omega_0)\|y\|_2^2 + \|y\|_\HH^2 + (\sigma_2 - \|\nabla\cdot v\|_\infty)\|y\|_2^2  
    & \leq 2\|f\|_2\|y\|_2 + 2C_{p,\Om} \|v\|_\infty \|y\|_{W^{1,p}(\Om)}^2   \notag\\
    & \leq 2\|f\|_2\|y\|_2 + \frac{2C_{p,\Om}\|v\|_\infty}{C(p,\Om)}\|y\|_\HH^2.
\end{align}
Further,
\begin{equation}   \label{eq:sigma12estimate}
    (\sigma_1 + \sigma_2 - \omega_0 - \|\nabla \cdot v \|_\infty)\|y\|_2^2 + \left(1-\frac{2C_{p,\Om}\|v\|_\infty}{C(p,\Om)}\right) \|y\|_\HH^2 \leq 2 \|f\|_2\|y\|_2.   
\end{equation}
Taking into account that $\HH\subset L^2(\Om)$ and \eqref{ec:vCritic}, where we can choose 
\begin{equation}
    C_0(p,\Om):=\frac{C(p,\Om)}{2C_{p,\Om}},
\end{equation}
we finally get
\begin{equation}
    (\sigma_1+\sigma_2-\omega_0 - \|\nabla\cdot v\|_\infty) \|y\|_2\leq  2\|f\|_2.   \notag
\end{equation}
Hence,
\begin{equation}  \label{eq:conclusion}
    \|y\|_2 \leq \frac{2\|f\|_2}{\abs{\sigma-\sigma_0}}, \;\;\text{where} \;\sigma_0:=\omega_0 + \|\nabla\cdot v\|_\infty.
\end{equation}

Now, let $\sigma_2<0$. We multiply scalarly in $L^2(\Om)$ the first equation by $y_2$ and
then add with the second equation multiplied by $-y_1$: Proceeding with the same calculus we obtain \eqref{eq:sigma12estimate} with $-\sigma_2 = \abs{\sigma_2}$ instead of $\sigma_2$, which implies again \eqref{eq:conclusion}.
\end{proof}

With the improvements above, under the hypotheses \eqref{ec:vCritic} on $v$, in the same manner as in the subcritical case, Theorem \ref{thm2} follows.


\begin{thebibliography}{}   \footnotesize
\bibitem{balinsky}  A. Balinsky, W. D. Evans, R. T. Lewis; \textit{The Analysis and Geometry of Hardy's Inequality}; Universitext, Springer, 2015.
\bibitem{baras} P. Baras, J.A. Goldstein; \textit{Remarks on the inverse square potential in quantum mechanics}, North-Holland Math. Stud., vol. 92 (C), pg. 31–35, 1984.
\bibitem{BarbuHinf2} V. Barbu, $H_\infty$ \textit{Boundary control with state feedback; The hyperbolic case}. Int. Ser. Numer. Math., vol. 107, pg. 141–148, 1992.
\bibitem{BarbuHinf3} V. Barbu, $H_\infty$ \textit{Boundary control with state feedback: the hyperbolic case}. SIAM J. Control Optim., vol. 33, pg. 684–701, 1995.
\bibitem{BarbuHinf1} V. Barbu, \textit{The $H_\infty$-Problem for Infinite-Dimensional Semilinear Systems}. SIAM J. Control Optim., vol. 33, iss. 4, 1995.
\bibitem{Barbu1} V. Barbu, \textit{Nonlinear Differential Equations of Monotone Type in Banach Spaces}. Springer, New York, 2010.
\bibitem{basar} T. Bașar, P. Bernhard; \textit{$H^\infty$-Optimal Control and Related Minimax Design Problems. A Dynamic Game Approach}, 2nd ed. Birkhäuser, Boston, Basel, Berlin, 1995.
\bibitem{bebernes} J. Bebernes, D. Eberly; \textit{Mathematical Problems from Combustion Theory}, Applied Math. Sci. 83. Springer-Verlag, New York, 1989.
\bibitem{brezis} H. Brezis; \textit{Functional Analysis, Sobolev Spaces and Partial Differential Equations}; Universitext, Springer, 2010.
\bibitem{brevaz} H. Brezis, J. L. Vazquez; \textit{Blow-up solutions of some nonlinear elliptic problems}; Rev. Mat. Complut., vol. 10, no. 2, pg. 443-469, 1997.
\bibitem{cazacuSIAM} C. Cazacu; \textit{Controllability of the heat equation with an inverse-square potential localized on the boundary}, SIAM J. Control Optim., vol. 52, no. 4, pg. 2055–2089, 2014.
\bibitem{curtain} R.F. Curtain, H.J. Zwart; \textit{An Introduction to Infinite Dimensional Linear Systems Theory}, Springer, 1995.
\bibitem{datko} R. Datko, \textit{Uniform asymptotic stability of evolutionary processes in a Banach space}, SIAM J. Math. Anal., vol. 3, pg. 428-445, 1972.
\bibitem{davies} E. B. Davies; \textit{Heat kernels and spectral theory}, Cambridge Tracts in Mathematics, vol. 92, Cambridge University Press, Cambridge, 1990.
\bibitem{evansCarte} L.C. Evans; \textit{Partial differential equations}, Graduate Studies in Mathematics, vol. 19, Amer. Math. Soc., Providence, RI, 1998.
\bibitem{ervedoza} S. Ervedoza; \textit{Control and stabilization properties for a singular heat equation with an inverse-square potential}, Comm. Partial Differential Equations, vol. 33, pg. 1996–2019, 2008.
\bibitem{fefferman} C. Fefferman; \textit{The uncertainty principle}, Bull. Amer. Math. Soc. (N. S.), vol. 9, no. 2, pg. 129–206, 1983.
\bibitem{azorero} J. P. Garcia Azorero; I. Peral Alonso; \textit{Hardy Inequalities and Some Critical Elliptic and Parabolic Problems}; J. differential equations, vol. 144, pg. 441-476, 1998.
\bibitem{gelfand} I. M. Gelfand; \textit{Some problems in the theory of quasi-linear equations}, Uspehi Mat. Nauk., vol. 14, no. 2, iss. 86, pg. 87–158, 1959.
\bibitem{glover} K. Glover; \textit{All optimal Hankel-norm approximations of linear multivariable systems and their $L_\infty$ error bounds}, Int. J. Control, vol. 39, pg. 1115–1193, 1984.
\bibitem{hardy2} G.H. Hardy, J.E. Littlewood, G. Pólya; \textit{Inequalities}; Reprint of the 1952 edition, Cambridge Mathematical Library, Cambridge University Press, 1988.
\bibitem{Lions} J. L. Lions, \textit{Optimal Control of Systems Governed by Partial Differential Equations}, Springer Berlin, Heidelberg, 1971.
\bibitem{kufner} A. Kufner, L. Malingranda, L.-E. Persson; \textit{The Hardy Inequality. About its history and some related results}; Pilsen, 2007.
\bibitem{marinoschi} G. Marinoschi; \textit{The $H^\infty$-control problem for parabolic systems with singular Hardy potentials}, ESAIM: Control Optim. Calc. Var., vol. 29, no. 73, 2023.
\bibitem{pazy} A. Pazy, \textit{ semigroups of Linear Operators and Applications to Partial Differential Equations},  Springer-Verlag, New York, 1983.
\bibitem{staffans} O.J. Staffans; \textit{On the distributed stable full information $H^\infty$ minimax problem}. Int. J. Robust Nonlinear Control, vol. 8, pg. 1255–1305, 1998.
\bibitem{vanKeulen1} B. Van Keulen, \textit{$H^\infty$-Control for Distributed Parameter Systems: A State-Space Approach}. Boston, MA, Birkhauser, 1993.
\bibitem{vanKeulen2} B. Van Keulen, M. A. Peters, R. Curtain, \textit{$H^\infty$-control with state feedback: the infinite dimensional case}. J. Math. Syst. Estim. Control, vol. 3, pg. 1-39, 1993.
\bibitem{vanderSchaft} A.J. van der Schaft; \textit{$L^2$ gain analysis of nonlinear systems and nonlinear state feedback $H^\infty$ control}, IEEE Trans. Automat. Control vol. 37 pg. 770–784, 1992.
\bibitem{vancost-zuazua} J. Vancostenoble, E. Zuazua; \textit{Hardy Inequalities, Observability, and Control for the Wave and Schrödinger Equations with Singular Potentials}, SIAM J. Math. Anal., 2009.
\bibitem{VazquezZogra} J. L. Vazquez, N. B. Zographopoulos, \textit{Functional Aspects of the Hardy Inequality: Appearance of a Hidden Energy}. J. Evol. Equ., vol. 12, iss. 3, 2011.
\bibitem{VazquezZuazua} J. L. Vazquez, E. Zuazua; \textit{The Hardy Inequality and the Asymptotic Behavior of the Heat Equation with an Inverse-Square Potential}, J. Funct. Anal., vol. 173, pg. 103-153, 2000.
\bibitem{zames} G. Zames, \textit{Feedback and optimal sensitivity: model reference transformations, multiplicative seminorms and approximate inverses}. IEEE Trans. Automat. Control, vol. 26, pg. 301–320, 1981.
\end{thebibliography}
\end{document}